\documentclass{amsart}

\usepackage[margin=1cm]{geometry}

\usepackage{amssymb,mathrsfs}

\usepackage[arrow,matrix]{xy}

\newcommand\HL{{\sf HL}}
\newcommand\HR{{\sf HR}}
\newcommand\Ho{{\sf H}}

\newcommand\x\times
\newcommand\ot\otimes
\newcommand\ox\otimes
\newcommand\G\Gamma
\newcommand\La\Lambda
\newcommand\la\lambda
\newcommand\g{{\mathcal L}}
\newcommand\fa{{\mathfrak a}}
\newcommand\gf{{\mathfrak g}}
\newcommand\fh{{\mathfrak h}}
\newcommand\fm{{\mathfrak m}}

\newcommand\fV{{\mathfrak R}}

\newcommand\ab{\mathsf{ab}}
\newcommand\ann{\mathsf{ann}}
\newcommand\sLIE{\mathscr{L\!I\!E}}

\newcommand\Lie{\mathsf{Lie}}

\newcommand\xto[1]{\xrightarrow[]{#1}}

\DeclareMathOperator\fLIE{\mathfrak{LIE}}
\DeclareMathOperator\fLB{\mathfrak{LB}}

\DeclareMathOperator\Hom{\mathsf{Hom}}

\newtheorem{De}{Definition}
\newtheorem{Th}[De]{Theorem}
\newtheorem{Pro}[De]{Proposition}
\newtheorem{Le}[De]{Lemma}
\newtheorem{Co}[De]{Corollary}

\begin{document}

\

\title{On certain class of Leibniz algebras}

\author[M. Jibladze]{Mamuka Jibladze}

\address{Razmadze Mathematical Institute, Tbilisi, Georgia}

\author[T. Pirashvili]{Teimuraz  Pirashvili}

\address{Department of Mathematics. University of Leicester. UK}
\maketitle

\section{Introduction} The paper has two goals. First goal is to give a complete proof (see  Theorem \ref{rva.15.26} ) of an old and unpublished  result of Mar\'ia Ronco \cite{ronco}, which describes free objects in the category of Leibniz algebras satisfying the identity:
$$[[x,x],y]=0.$$
We call such algebras as Ronco algebras. Our interest with theory of Ronco algebras comes from the fact that it is a minimal quotient of the theory of Leibniz algebras, which is a linear extension of the theory of Lie algebras in the sense of \cite{JP}.  Our second goal is  to extend relationship between symmetric Leibniz algebras and Lie $\Sigma\mu$-algebras  (see
 \cite[Theorem 11 ]{JPS})  to Ronco algebras and so called Lie $\mu$-algebras, see Theorem \ref{14.15.30}. 

\section{Preliminaries on  second Leibniz homology} In what follows $K$ is a field of characteristic not equal 2.  All vector spaces  are taken over $K$.  Moreover, instead $\ot_K$ and ${\sf Hom}_K$ we will simply write $\ot$ and $\Hom$.

 Recall that \cite{UL} a {\bf Leibniz algebra} is a vector spaces  $\g$, equipped with an operation $[-,-]:\g\ox\g\to\g$ such that 
\begin{equation}\label{R}
[x,[y,z]]=[[x,y],z]-[[x,z],y].
\end{equation}

 We refer the reader to \cite{UL},\cite{Lrep},\cite{HL},\cite{conj}, \cite{perfect},\cite{subcpx}   for more on Leibniz algebras, Leibniz homology, Leibniz representations and some conjectures about Leibniz algebras and Leibniz homology.

The following  identities are consequences of   (\ref{R}):
\begin{equation}\label{an}[x,[y,y]]=0,  \quad [x,[y,z]]+[x,[z,y]]=0.\end{equation}

The category of Leibniz algebras is denoted by $\fLB$.  It is clear that  Lie algebras are those Leibniz algebras $\g$ for which the identity  $[x,x]=0$ holds. Denote by $\g^{\sf ann}$ the subspace of $\g$ generated by elements $[x,x]$, $x\in \g$. Then $\g^{\sf ann}$ is a two-sided ideal of $\g$  \cite{UL} and the quotient $\g_\Lie$ is a Lie algebra. Moreover the assignment $\g\mapsto \g_\Lie$ defines the left adjoint functror to the inclusion of the category $\fLIE$ of Lie algebras in $\fLB$ \cite{UL}. 
It follows that for any Leibniz algebra $\g$ one has  an exact sequence
\begin{equation}\label{s2lie}{\sf Sym}^2(\g)\xto{\mu} \g\to \g_\Lie\to 0\end{equation}
where ${\sf Sym}^2$ denotes the second symmetric power and $\mu(x\odot y)=[x,y]+[y,x], x,y\in\ g.$ 
Here $x\odot y$ denotes the image of $x\otimes y$ in ${\sf Sym}^2(\g)$ under the canonical map $\g\otimes\g\to {\sf Sym}^2(\g)$.

We refer to \cite{UL} for definition of the  homology and cohomology of Leibniz algebras. Here we will need mostly the first and the  second homologies, which can be defined by
$${\HL}_1(\g)=\g_{\sf ab}=\g/[\g,\g].$$
$${\HL}_2(\g)=\frac{Ker([-,-]:\g\ot\g\to \g)}{{Im(d:\g\ot \g\ot \g\to \g\ot\g)}}$$
where  $$d(x\ot y \ot z)=[x,y]\ot z-[x,z]\ot y-x\otimes [y,z]$$
Recall that \cite{UL} for any vector space $\fa$ considered as an abelian Leibniz algebra, the vector space $\Hom(\HL_2(\g),\fa)$ classifies all central extensions of $\g$ by $\fa$. From this fact and the classical Yoneda Lemma we obtain the following well-known result.

\begin{Le}\label{1} For any Leibniz algebra $\g$ there exists a central extension
$$0\to \HL_2(\g)\to \widehat{\g}\xto{p} \g\to 0$$ with properties:

(i) The homomorphism $p$ induces an isomorphism $\widehat{\g}_{\sf ab}\cong \g_{\sf ab}$, where as usual $\g_{\sf ab}=\g/[\g,\g]$.

(ii)  For any central extension $$0\to \fa\to \fh\to \g\to 0$$ there exist a  commutative diagram of Leibniz algebras and Leibniz algebra homomorphisms
$$\xymatrix{&& \widehat{\g}\ar[d]_{f} \ar[r]^p& \g\ar[d]^{id_\g}\ar[r]&0\\
0\ar[r]&\fa \ar[r]&\fh\ar[r]&\g\ar[r]&0}$$ 

(iii) If $f':\widehat{\g}\to \fh$ also  fits in the commutative diagram, then there exist a unique linear map $\gamma:\g_{\sf ab} \to \fa$ such that $f_1=f+\gamma\circ p$. In particular there exist a unique $\alpha$ in the commutative diagram 
$$\xymatrix{0\ar[r]& \HL_2(\g)\ar[d]^\alpha\ar[r]& \widehat{\g}\ar[d]_{f} \ar[r]^p& \g\ar[d]^{id_\g}\ar[r]&0\\
0\ar[r]&\fa \ar[r]&\fh\ar[r]&\g\ar[r]&0}$$ 
\end{Le}

In this paper we are mostly interested in the case, when $\gf$ is a Lie algebra. Leibniz homology of Lie algebras was investigated in \cite{HL}. We will need some results from that paper. In order to state them we set
$$\HR_0(\gf)=\Ho_0(\gf,{\sf Sym^2}(\gf)).$$
In other words $\HR_0(\gf)$ is the quotient of ${\sf Sym^2}(\gf)$ by the relation $x\odot [y,z]=[x,y]\odot z$.

Observe that for abelian Lie algebra $\gf$ one has $\HR_0(\gf)={\sf Sym^2}(\gf)$. Hence for any Lie algebra $\gf$ the abelization map $\gf\to\gf_{\sf ab}$ induces the homomorphism
$$\HR_0(\gf)\to{\sf Sym}^2(\gf_{\sf ab})$$

\begin{Le} Let  $\gf$ be a  free nil$_2$-class Lie algebra generated by  a vector space $V$. That is 
$$\gf=V\bigoplus \Lambda^2(V), \  [u,v]=u\wedge  v,  \ {\rm and}\  [v,w]=0, [w,w']=0, u,v\in V, w,w'\in \Lambda^2(V), $$
Then one has en exact sequence 
$$0\to \Lambda ^3(V)\to \HR_0(\gf)\to {\sf Sym}^2(V)\to 0.$$
\end{Le}
\begin{proof} Clearly 
$$ {\sf Sym}^2(\gf)= {\sf Sym}^2(V)\bigoplus V\ot \Lambda^2(V)\bigoplus  {\sf Sym}^2(\Lambda^2(V))$$
Now, take $w,w'\in \Lambda^2(V)$, assume $w'=u\wedge v$. Then
$$w \odot w'=w\odot [u,v]=[w,u]\odot v=0$$
Thus  the image of ${\sf Sym}^2(\Lambda^2(V))$ in $\HR_0(\g)$ is zero. We also have
$$u\odot v\wedge v'=u\odot [v,v']=v'\odot u\wedge v$$
where $u,v,v'\in V$. It follows that the kernel of the canonical map $\HR_0(\gf)\to {\sf Sym}^2(V)$ is the same as 
$$\frac{ V\ot \Lambda^2(V)}{u\odot v\wedge v'=v'\odot u\wedge v}$$
The last question is isomorphic to $\Lambda ^3(V)$ via the map $u\odot v\wedge v'\mapsto u\wedge v\wedge v'$. To see the last statement it suffices to check this fact in dimensions 1,2,3 because the functors in the questions are cubical.
\end{proof}

\begin{Pro}\label{pirahl} Let $\gf$ be a Lie algebra. Then following holds:

i)  There is a natural isomorphism $$\HL_2(\g)\cong \Ho_1(\g,\g^{ad}),$$
where the right hand side is the classical Lie algebra homology with coefficients in the adjoint representation.

ii)  There is a natural homomorphism
$$\HR_0(\g)\to \Ho_1(\g,\g^{ad})$$
which is an isomorphism if $\Ho_2(\g)=0=\Ho_3(\g)$. 
\end{Pro}

\begin{proof} All results were proved in \cite{HL}.  
\end{proof}
\begin{Co} For a free Lie algebra $\g=\bigoplus_{n\geq 1}{\sf Lie}_n(V)$ one has
$${\HL}_2(\g)\cong \Ho_1(\g,\g^{ad})\cong \HR_0(\g)\cong \bigoplus_{n\geq 2}Ker(Lie_{n-1}(V)\ot V\to Lie_n(V))$$
\end{Co}
\begin{proof} Since Lie algebra homology vanishes on free Lie algebras in dimensions $\geq 2$ all statements except the last isomorphism follows from  Proposition \ref{pirahl}.
To prove the last isomorphism, observe that if $\g=\bigoplus_{n\geq 1}{\sf Lie}_n(V)$ is a free Lie algebra and $M$ is a $\g$-module, then any derivation $\g\to M$ uniquely defined by  a linear map $V\to M$. Hence one has an exacts sequence:
$$M\xto{\partial} {\sf Hom} (V,M)\to H^1(\g,M)\to 0$$
where $\partial(m)(v)=[v,m]$. By duality we also have an exact sequence
$$0\to H_1(\g,M)\to V\ot M\xto{\delta} M$$
where $\delta(v\ot m)=[v,m].$ This implies the isomorphism $\Ho_1(\g,\g^{ad})\cong \bigoplus_{n\geq 2}{\sf R}_n(V)$.

\end{proof}
For higher  Leibniz homology of a free Lie algebra  we refer to  \cite{HL}.

\section{Ronco algebras}
A {\bf Ronco algebra}  is a  Leibniz algebra $X$ if it is   a central extension of a  Lie  algebra. Thus if there exists a central  extension
\begin{equation} \label{ce}0\to A\to \g\to \fh\to 0,\end{equation} 
of Leibniz algebras, for which $\fh$ is a Lie algebra.
\begin{Le} A  Leibniz algebra $\g$ is a Ronco algebra, iff
\begin{equation}\label{mr}
[[x,x],y]=0.
\end{equation}
holds for all $x,y\in \g$. If this the case, then the following identity holds:
\begin{equation}\label{mrv} [[x,y],z]+[[y,x],z]=0\end{equation}
\end{Le}
\begin{proof} We have a short exact sequence
$$0\to \g^{\sf ann}\to \g\to \g_\Lie\to 0.$$
By the relations (\ref{an}) we have $[\g,\g^{\sf ann}]=0$. Thus if $[[x,x],y]=0$ holds, then $\g^{\ann}$ is a central subalgebra of $\g$ and hence $\g$ is a Ronco algebra. Conversely, assume $\g$ is a Ronco algebra  and (\ref{ce}) is a central extension with $\fh\in \sLIE$. Then for any $x\in \g$ we have $[x,x]\in A$ and hence $[[x,x],y]\in [A,\g]=0$.
\end{proof}

Thus any symmetric Leibniz algebra \cite{JPS} is a Ronco algebra. The collection of all Ronco algebras is denoted by $\fV$. Thus $\fV$ is a variety of Leibniz algebras. Hence it posses free algebras.
\begin{Le}\label{ffv} Take a free Lie algebra $\g$ on a vector space $V$, that is $\g=\bigoplus_{n\geq 1} {\sf Lie}_n(V)$, ${\sf Lie}_1(V)=V$ and consider the central extension 
$0\to \HL_2(\g)\to \widehat{\g}\xto{p} \g\to 0$ as in Lemma \ref{1}. Choose a linear section $s:\g\to \widehat{\g}$ of $p$. If we identify $s(V)$ and $V$, then $\widehat{\g}$ is free object in $\fV$ generated by $V$. 
\end{Le}
\begin{proof}
Take $\fh\in\fV$ and a linear map $f:V\to \fh$. We have to show that $f$ has a unique extension (still denoted by $f$) $\widehat{\g}\to \fh$ which is Leibniz algebra homomorphism. First we show the uniqueness. Assume $f$ has two extensions $f_1,f_2:\widehat{\g}\to \fh$. Since $(-)_\Lie$ is a functor we obtain corresponding homomorphisms of Lie algebras $(f_1)_\Lie,(f_2)_\Lie:\g\to \fh_\Lie$. Since both of them agree on $V$ and $\g$ is free on $V$, we see that $(f_1)_\Lie=(f_2)_\Lie$. Denote this common value by $f_\Lie$. Then we have $f_2=f_1+\gamma\circ p$, thanks to Lemma \ref{1} for  a linear map $\gamma:\g_{\sf ab}\to \fa$. Since $f_1=f_2$ on $V$ and the composite $V\to \g\to \g_{\ab}$ is an isomorphism, we see that $\gamma=0$. Thus $f_=f_2$ on $\widehat{\g}$ and the uniqueness follows.

For existence, observe that $f$ has  a unique extension to  a homomorphism of Lie algebras $f':\g\to \fh_\Lie$, because $\g$ is free Lie algebra. Since $\fh\in\fV$ the short exact sequence 
$$0\to \fh^{ann} \to \fh \to \fh_\Lie\to 0$$
is a central extension of the Lie algebra $\fh_\Lie$. Thus by Lemma \ref{1} we can form  the following commutative diagram of central extensions
$$\xymatrix{0\ar[r] & \HL_2(\g)\ar[r]\ar[d] & \widehat{\g} \ar[r]^p\ar[d]^{g'} & \g \ar[r]\ar[d]^{id} & 0\\
0\ar[r] & \fa\ar[r]\ar[d]^{id} & \fh' \ar[r]\ar[d]^g & \g \ar[r]\ar[d]^{f'} &0 \\
0\ar[r] & \fa\ar[r]& \fh \ar[r] & \fh_\Lie \ar[r] &0
}$$
where $\fh'$ is the pull-back in the corresponding diagram. It follows that for any $v\in V$, we have $f(v)-gg'(v)=\in \fa$. Since $V\cong \g_{\sf ab}$ is an isomorphism, we have a unique linear map $\gamma:\g_{\sf ab}\to \fa$ for which $\gamma(v)=f(v)-gg'(v)$, $v\in V$. It follows that $gg'+\gamma\circ p$ is a Leibniz homomorphism, which extends the map $f:V\to \fh$.
\end{proof}

\begin{Co} The theory ${\sf Ronco}$ of Ronco algebras fits in the linear  extension of algebraic theories \cite{JP}: 
$$0\to \HR_0\to {\sf Ronco} \to {\sf Lie}\to 0$$
\end{Co} 



Recall that \cite{UL} the free Leibniz algebra generated by  a vector space $V$ is 
$${\sf Leib}(V)=V\bigoplus V^{\ot 2} \bigoplus V^{\ot 3}\bigoplus \cdots =\bigoplus_{n\geq 1} V^{\ot n},$$
where the bracket is uniquely defined by the property $$[\omega,v]=\omega \ot v$$
Here $\omega\in V^{\ot n}$ and $v\in V$.  

The following result first was obtained by Mar\'ia Ronco \cite{ronco} in her unpublished notes written arround 1995.

\begin{Th}\label{rva.15.26} Let $V$ be a vector space. Then there exists a unique bracket $-,-]$ on 
$${\fV}(V)= V\bigoplus_{n\geq 1}  {\sf Lie}_{n}(V)\ot V$$
such that the canonical map ${\sf Leib}(V)\to \fV(V)$, induced by the projection $$V^{\ot n}\to {\sf Lie}_{n-1}(V)\ot V, v_1\ot \cdots \ot v_n\mapsto \{\{ v_1,v_2\},\cdots v_{n-1}\}\ot v_n$$ preserves the bracket. Moreover ${\fV}(V)$ is a free object Ronco algebra  generated by $V$. Here $\{-,-\}$ is the bracket on ${\sf Lie}(V)$.
\end{Th}
\begin{proof}  We keep the notation from Lemma \ref{ffv}. Thus $\g$ denotes the free Lie algebra generated by $V$ and $\widehat{\g}$ is the free algebra in $\fV$ generated by $V$. Then we have a canonical epimorphism $f:{\sf Leib}(V)\to \widehat{\g}$. Thanks to Lemma (\ref{m.r.}) below the map $f$ factors through the map
$$g:\fV(V)\to \widehat{\g}.$$
By construction of $\fV(V)$ and by Lemma \ref{ffv} this map fits in the following diagram

$$\xymatrix{0\ar[r] & \bigoplus_{n\geq 1} {\sf R}_n(V) \ar[r]\ar[d] & \fV(V) \ar[r]^p\ar[d]^{g} & \g \ar[r]\ar[d]^{id} & 0\\
0\ar[r]  & \HL_2(\g)\ar[r]&\ar[r] \widehat{\g} \ar[r]& \g \ar[r] &0}
$$
By Proposition \ref{pirahl} the first vertical map is an isomorphism. Hence $g$ is also an isomorphism and the reslult follows.

\end{proof}

\begin{Le}\label{m.r.} Let $f:\g\to \g'$ be  a Leibniz algebra homomorphism. If $\g'\in \fV$ and $x\in Ker(\g\to \g_{Lie})$, then $f([x,y])=0$ for all $y\in\g$.
\end{Le}
\begin{proof} By the exact sequence (\ref{s2lie}), $x$ is a linear combination of elements of the form $[u,v]+[v,u]$. Hence the result follows from the identity (\ref{mrv}).
\end{proof}

\section{$\mu$-algebras}
\begin{De}\label{detheta} A  $\mu$-algebra is a  vector space  $\fm$  equipped with two binary operations 
$\mu:\fm\ox\fm\to \fm$  and $\lbrace -,-\rbrace:\fm\ox\fm\to \fm$ such that the following identities hold
\begin{itemize}
	
	\item [i)] $xy=yx$,
	
	\item [ii)] $x(yz)=0=(xy)z$
	
	\item [iii)] $ \lbrace x,x\rbrace =0$ and hence $ \lbrace x,y\rbrace  + \lbrace y,x\rbrace =0$,
	
	\item [iv)] $ \lbrace xy,z\rbrace =0$ 	

           \item [v)] $\lbrace x, \lbrace y,z\rbrace\rbrace+\lbrace z, \lbrace x,y\rbrace\rbrace+\lbrace y, \lbrace z,x\rbrace\rbrace=
x\lbrace y,z\rbrace$
		
	\end{itemize}
Here $xy=\mu(x\ox y)$.
\end{De}

A $\mu$-algebra is \emph{symmetric} if $x \lbrace y,z\rbrace=0$. So, in this case $ \lbrace -,-\rbrace$ is a Lie algebra structure. These algebras  were studied in \cite{JPS}.

\begin{Le}\label{skew} The function $(x,y,z)\mapsto x \lbrace y,z\rbrace$ is skew-symmetric. In particular, 
$x \lbrace x,y\rbrace =0$ and hence $x\lbrace y,z\rbrace+y\lbrace x,z\rbrace=0$.
\end{Le}
\begin{proof} This is a direct consequence of v).
\end{proof}

\begin{Th}\label{14.15.30} The category of $\mu$-algebras is equivalent to the category of Ronco algebras, under this equivalence symmetric $\mu$-algebra corresponds to symmetric Leibniz algebras.
\end{Th}

The theorem is  a trivial consequence of Propositions \ref{9} and \ref{10}.
\begin{Pro}\label{9} For  a Ronco algebra $\g$, we put:
$$2 \lbrace x,y\rbrace= [x,y]-[y,x],$$
$$2xy=[x,y]+[y,x].$$
{\rm (}Thus $[x,y]=\lbrace x,y\rbrace +xy$.{\rm )} Then $\g$ together with operations $\lbrace -,-\rbrace$ and
$\mu(x,y)=xy$ is a $\mu$-algebra.
\end{Pro}
\begin{proof} The relations i) and iii) are obvious. By the relation (\ref{an}) and  (\ref{mrv}) we have 
$$[xy,z]=0 \quad {\rm and} \quad [x,yz]=0.$$
Hence
$$2(xy)z=[xy,z]+[z,xy]=0+0=0.$$
Similarly,
$$2x(yz)=[x,yz]+[yz,x]=0$$
and ii) follows. Quite similarly, 
$$2 \lbrace xy,z\rbrace =[xy,z]-[z,xy]=0-0=0$$
and iv) follows.

Next, we have 
$$4x \lbrace y,z\rbrace=[x,[y,z]-[z,y]]+[[y,z]-[z,y],x]$$
By (\ref{an}) and (\ref{mrv}) we can rewrite:
$$4x \lbrace y,z\rbrace=2[x,[y,z]]+2[[y,z],x]=2[[x,y],z]-2[[x,z],y]+2[[y,z],x].$$
Let us use  (\ref{mrv}) once more, to obtain
\begin{equation}\label{ps}
2x \lbrace y,z\rbrace=[[x,y],z]+[[z,x],y]+[[y,z],x] =\sum_{cyclic}[[x,y], z] \end{equation}

Next, we consider
$$4 \lbrace x, \lbrace y,z\rbrace\rbrace=[x,[y,z]-[z,y]]-[[y,z]-[z,y],x]=2[x,[y,z]+2[[z,y],x]$$
and hence
$$2\lbrace x, \lbrace y,z\rbrace\rbrace=[[x,y],z]-[[x,z],y]+[[z,y],x]=[[x,y],z]+[[z,x],y]+[[z,y],x].$$
So, we proved
\begin{equation}\label{ms}2\lbrace x, \lbrace y,z\rbrace\rbrace=[[x,y],z]+[[z,x],y]-[[y,z],x]\end{equation}
This implies
$$2\sum_{cyclic}\lbrace x, \lbrace y,z\rbrace\rbrace= \sum_{cyclic}[[x,y],z]$$
This together  (\ref{ps}) implies v).

\end{proof}

\begin{Pro}\label{10} Let $\fm$ be a $\mu$-algebra. We put
$$[x,y]=\lbrace x,y\rbrace +xy.$$
Then $[-,-]$ defines a Ronco algebra   structure on $\fm$.
\end{Pro}
\begin{proof} We have
$$[x,[y,z]=x(yz+\lbrace y,z\rbrace)+\lbrace x,yz+ \lbrace y,z\rbrace\rbrace= x\lbrace y,z\rbrace+\lbrace x,\lbrace y,z\rbrace\rbrace
$$
Quite similarly
$$[[x,y],z]= (xy+\lbrace x,y\rbrace) z +\lbrace xy+\lbrace x,y\rbrace,z\rbrace =  \lbrace x,y\rbrace  z +\lbrace \lbrace x,y\rbrace,z\rbrace$$
and
$$[[x,z],y]=  \lbrace x,z\rbrace  y +\lbrace \lbrace x,z\rbrace,y\rbrace$$
Hence
\begin{align*}
[[x,y],z]-[[x,y],z]+[[x,z],y]=&x\lbrace y,z\rbrace-\lbrace x,y\rbrace  z  +\lbrace x,z\rbrace  y \\
+&\lbrace x,\lbrace y,z\rbrace\rbrace -\lbrace \lbrace x,y\rbrace,z\rbrace +\lbrace \lbrace x,z\rbrace,y\rbrace
\end{align*}
\end{proof}
By i) of Definition \ref{detheta} and Lemma \ref{skew} the first line equals to $-x\lbrace y,z\rbrace$. By iii) of Definition \ref{detheta} the second line equals to 
$$\lbrace x,\lbrace y,z\rbrace\rbrace +\lbrace z, \lbrace x,y\rbrace\rbrace +\lbrace y, \lbrace z,x\rbrace\rbrace$$
which is equal to $x\lbrace y,z\rbrace$ thanks to vi) of Definition \ref{detheta}. It follows that $[[x,y],z]-[[x,y],z]+[[x,z],y]=0$ and hence $\fm$ is a Leibniz algebra. We also have
$[x,x]=xx$, hence $[[x,x],y]=[xx,y]=0$ and the result follows. 


\end{document}